\documentclass[a4paper]{article}
\usepackage{template}

\newcommand{\Formalsch}{\mathrm{FSch}}
\newcommand{\nFormalsch}{\mathrm{FSch}^{\mathrm{qt}}}
\newcommand{\smFormalsch}{\mathrm{FSch}^{\mathrm{sq}}}
\newcommand{\LocSpaces}{\mathrm{LocSpaces}}
\newcommand{\complproj}{(\mathrm{Proj}^{\omega_1,\mathrm{cpl}}_{R_I^\wedge})^{\mathrm{op}}}

\newcommand{\ideals}{\mathrm{tt}\text{-}\mathrm{Cat}^{\mathrm{wcg}}}
\newcommand{\largett}{\mathrm{tt}\text{-}\mathrm{Cat}^{\mathrm{mcg}\text{-}\mathrm{fun}}}
\newcommand{\higherlargett}{\mathrm{tt}\text{-}\mathrm{Cat}^{\mathrm{mcg}\text{-}\mathrm{fun}}_{\infty}}
\newcommand{\dualpr}{\mathrm{Cat}^{\mathrm{dual}}_{\mathrm{st}}}
\newcommand{\tors}[1]{D_{{#1}\text{-}\mathrm{tors}}}
\newcommand{\Dtors}{D_{\mathrm{tors}}}
\newcommand{\prcpt}{\mathrm{Pr}^{\mathrm{L}}_{\mathrm{st},\omega}}
\newcommand{\pr}{\mathrm{Pr}^{\mathrm{L}}_{\mathrm{st}}}

\newcommand{\CAlg}{\mathrm{CAlg}}
\newcommand{\Nuc}{\mathrm{Nuc}}
\newcommand{\eNuc}{\mathrm{Nuc}^\mathrm{Ef}}
\newcommand{\csNuc}{\mathrm{Nuc}^{\mathrm{CS}}}
\newcommand{\Perf}{\mathrm{Perf}}

\newcommand{\fld}{\mathrm{Fld}}

\DeclareMathOperator*{\colim}{colim}

\newcommand{\Hom}{\operatorname{Hom}}
\newcommand{\sHom}{\operatorname{\mathcal{H}om}}
\newcommand{\RHom}{\operatorname{RHom}}
\newcommand{\End}{\operatorname{End}}
\newcommand{\Map}{\operatorname{Map}}

\newcommand{\Ker}{\mathrm{Ker}}

\newcommand{\Spc}{\mathrm{Spc}}
\newcommand{\Spec}{\mathrm{Spec}}
\newcommand{\Spf}{\mathrm{Spf}}

\newcommand{\mR}{\mathrm{R}}

\newcommand{\mcg}{\mathrm{mcg}}
\newcommand{\ang}[1]{\langle{#1}\rangle}
\newcommand{\oang}[1]{\langle{#1}\rangle_{\otimes}}
\newcommand{\Supp}{\mathrm{Supp}}
\newcommand{\sO}{\mathcal{O}}
\newcommand{\op}{\mathrm{op}}
\newcommand{\Idem}{\mathrm{Idem}}
\newcommand{\os}{\overset}
\newcommand{\Kos}{\mathrm{Kos}}

\newcommand{\Frak}{\mathfrak}
\newcommand{\Ho}{\mathrm{Ho}}
\newcommand{\id}{\mathrm{id}}
\newcommand{\Id}{\mathrm{Id}}
\newcommand{\rig}{\mathrm{rig}}
\newcommand{\Ind}{\mathrm{Ind}}
\newcommand{\trdeg}{\mathrm{tr.deg}}

\newcommand{\nat}{\mathbf{N}}
\newcommand{\intg}{\mathbf{Z}}

\title{Reconstruction of Formal Schemes from Categories of Nuclear Modules}
\author{Hisato Matsukawa\thanks{Department of Mathematics, Faculty of Science, Hokkaido University\\Kita 10, Nishi 8, Kita-Ku, Sapporo, Hokkaido, 060-0810, Japan\\Email:matsukawa.hisato.f4@elms.hokudai.ac.jp,hmnr0211@gmail.com\\
Keywords: formal schemes, nuclear modules, Balmer spectrum.}}
\date{}

\begin{document}
\maketitle

\begin{abstract}
  We provide a partially functorial and constructive reconstruction procedure for formal schemes from symmetric monoidal categories of nuclear modules.
  More precisely, for a formal scheme \(\Frak{X}\), we show that the derived category of torsion modules \(\Dtors(\Frak{X})\) can be recovered as the maximal strongly compactly generated localizing tensor ideal of Efimov's category \(\eNuc(\Frak{X})\), and similarly for the Clausen--Scholze category \(\csNuc(\Frak{X})\).
  Combining this with the Balmer spectrum, we reconstruct \(\Frak{X}\) from the corresponding symmetric monoidal category of nuclear modules.
  Moreover, for formal schemes topologically of finite type over a field or over \(\mathbb{Z}\), the contravariant functor \(\Frak{X} \mapsto \eNuc(\Frak{X})\) is fully faithful; in the affine case, the analogous statement holds for \(\Frak{X} \mapsto \csNuc(\Frak{X})\).
  \end{abstract}

\tableofcontents

\section{Introduction}

The aim of this paper is to establish a reconstruction result for formal schemes from Efimov's and Clausen--Scholze categories of nuclear modules.

For a formal scheme \(\Frak{X}\), let \(\eNuc(\Frak{X})\) denote Efimov's category of nuclear modules \cite{Efimov2025LocalizingInvariantsInverseLimits}, and let \(\csNuc(\Frak{X})\) denote the Clausen--Scholze category of nuclear modules in the sense of \cite{Scholze2026LecturesAnalyticGeometry,And23,AM24}.

For an \(\infty\)-category \(C\), we denote its homotopy category by \(\Ho(C)\).

\begin{thm}[Corollary \ref{cor:identical_formal_schemes}]
Let \(\Frak{X},\Frak{Y}\) be Noetherian formal schemes.
\begin{enumerate}
\item If there is a symmetric monoidal triangulated equivalence \(\Ho(\eNuc(\Frak{X}))\cong \Ho(\eNuc(\Frak{Y}))\), then \(\Frak{X}\cong \Frak{Y}\).
\item If \(\Frak{X},\Frak{Y}\) are affine, and there is a symmetric monoidal triangulated equivalence \(\Ho(\csNuc(\Frak{X}))\cong \Ho(\csNuc(\Frak{Y}))\), then \(\Frak{X}\cong \Frak{Y}\).
\end{enumerate}
\end{thm}

Moreover, we give a partially functorial and constructive reconstruction.

The proof is roughly as follows.
First, we recover the derived category \(\Dtors(\Frak{X})\) of torsion modules from the nuclear categories \(\eNuc(\Frak{X})\) and \(\csNuc(\Frak{X})\).
More precisely, \(\Dtors(\Frak{X})\) is recovered as a certain maximal strongly compactly generated localizing ideal.
Next, we recover the formal scheme \(\Frak{X}\) itself from \(\Dtors(\Frak{X})\) by using a slight modification of Balmer's tensor triangular spectrum.

We also discuss a second reconstruction method in the spirit of Liu's functorial spectrum \cite{Liu2011FunctorsTriangulatedTensorCategories}.
Instead of first extracting \(\Dtors(\Frak{X})\), this method uses field-valued points of the nuclear category.
It yields the following full faithfulness result.

Let \(k\) be a commutative ring.
Let \(\mathrm{fin.FSch}_k\) be the category of formal schemes topologically of finite type over \(k\), and let \(\mathrm{fin.aff.FSch}_k\) be the category of affine formal schemes topologically of finite type over \(k\).

Let \(\dualpr\) be the \(\infty\)-category of dualizable presentable stable \(\infty\)-categories and functors which admit two successive right adjoints.
Let \(\CAlg_k(\dualpr)\) be the \(\infty\)-category of commutative algebra objects in \(\dualpr\) under \(D(k)\).

\begin{thm}
  Let \(k=\intg\) or a field.
  The following functors are fully faithful:
  \begin{align*}
    (\mathrm{fin.FSch}_k)^\op&\to \CAlg_k(\dualpr):\Frak{X}\mapsto \eNuc(\Frak{X}),\\
    (\mathrm{fin.aff.FSch}_k)^\op&\to \CAlg_k(\dualpr):\Frak{X}\mapsto \csNuc(\Frak{X}).
  \end{align*}
\end{thm}

We prove this in a more general form in Theorem \ref{thm:fully_faithfulness_strong_ver}.

In Section 2, we introduce notation and recall basic properties of triangulated categories, tensor triangulated categories, and categories associated with formal schemes.
In Section 3, we prove that the maximal strongly compactly generated localizing ideals of \(\eNuc(\Frak{X})\) and \(\csNuc(\Frak{X})\) are given by the derived category of torsion modules \(\Dtors(\Frak{X})\).
In Section 4, we introduce a modification of Balmer's tensor triangular spectrum and apply it to \(\Dtors(\Frak{X})\).
Since the unit object of \(\Dtors(\Frak{X})\) is not compact, we need to slightly modify the usual definition.
We also prove the functoriality needed later.
In Section 5, we give a partially functorial and constructive reconstruction of formal schemes.
In Section 6, we give another approach to reconstruction and, as a consequence, prove the full faithfulness of the functors \(\eNuc\) and \(\csNuc\).

\subsection*{Assumptions and Notation}

We assume the existence of a Grothendieck universe.

All rings are assumed to be commutative and unital.

All formal schemes are assumed to be quasi-compact and quasi-separated and to admit a finitely generated ideal of definition locally.

Unless otherwise specified, all triangulated categories are assumed to admit small coproducts and to be well generated.

By an \(\infty\)-category, we mean an \((\infty,1)\)-category.
We regard ordinary categories and \((2,1)\)-categories as \(\infty\)-categories in the usual way.

For a functor \(F\), we denote its right adjoint, when it exists, by \(F^\mR\).

\section{Preliminaries}

\subsection{Tensor triangulated categories}
In this section, we recall some basic facts on triangulated categories and tensor triangulated categories.
See \cite{Nee01} and \cite{BalmerSpectrumPrimeIdeals} for details.

Let \(T,T'\) be triangulated categories.
A functor \(F:T\to T'\) is called continuous if it is exact and preserves small coproducts.
Such a functor has an exact right adjoint.
If its right adjoint is also continuous, then \(F\) is called strongly continuous.
Strongly continuous functors preserve \(\kappa\)-compact objects for every regular cardinal \(\kappa\).
The full subcategory of \(\kappa\)-compact objects is denoted by \(T^\kappa\).

Let \(I\subset T\) be a thick subcategory.
It is called a localizing subcategory if it is closed under small coproducts.
A localizing subcategory \(I\) is well-generated as a triangulated category if and only if it is generated by a set of objects of \(T\).
In this case, the Verdier quotient \(L_I:T\to T/I\) is continuous and has a fully faithful right adjoint \(L_I^{\mR}\).
The essential image of \(L_I^{\mR}\) is the subcategory \(I^\perp\) of \(I\)-local objects, that is, objects \(a\in T\) such that \(\Hom(I,a)=0\).
The inclusion \(\iota_I:I\to T\) also has a right adjoint \(\iota^{\mR}_I\).
The functor \(\iota_I\) is strongly continuous if and only if \(L_I\) is strongly continuous.
In this case, \(I^\kappa=T^\kappa\cap I\).

The localizing subcategory generated by a subset \(S\subset T\) is denoted by \(\ang{S}\).
If each \(s\in S\) is compact in \(T\), then \(\ang{S}\) is called compactly generated.
If, moreover, \(S\) is finite, then \(\ang{S}\) is called finitely compactly generated.
In this case, the inclusion \(\iota_I:I\to T\) is automatically strongly continuous, and \((T/I)^\kappa=\Idem(T^\kappa/I^\kappa)\), where \(\Idem\) denotes idempotent completion.

A tensor triangulated category, or tt-category for short, is a triangulated category \(T\) equipped with a symmetric monoidal structure \(-\otimes-:T\times T\to T\) which is exact in each variable and preserves small coproducts in each variable.
The unit object is denoted by \(1_T\).
A homomorphism between tt-categories is a symmetric monoidal continuous functor.

A localizing subcategory \(I\) of a tt-category \(T\) is called a localizing ideal if, for all \(a\in T\) and \(b\in I\), one has \(a\otimes b\in I\).
For a subset \(S\subset T\), the localizing ideal generated by \(S\) is denoted by \(\oang{S}\).
If each \(s\in S\) is compact, then \(\oang{S}\) is called (weakly) compactly generated.
If, moreover, \(S\) is finite, then \(\oang{S}\) is called (weakly) finitely compactly generated.
If \(I\) is a well-generated localizing ideal, then the Verdier quotient \(T/I\) inherits a natural tt-structure such that \(L_I:T\to T/I\) is symmetric monoidal. 
A localizing ideal \(I\) is called a smashing ideal if the inclusion \(\iota_I:I\to T\) is \(T\)-linear, that is, if \(\iota_I(\iota^{\mR}_I(a)\otimes b)\cong a\otimes \iota_I(b)\) for all \(a\in T\) and \(b\in I\).
For a smashing ideal \(I\), the inclusion \(\iota_I\) is always strongly continuous.
Moreover, \(I^\perp\) is a localizing ideal and \(\iota_I^\mR\) is identified with the quotient functor \(T\to T/I^\perp\).
In particular, \(I\) inherits a natural tt-structure such that \(\iota_I^\mR\) is symmetric monoidal.

A left idempotent in a tt-category \(T\) is a morphism \(\gamma:e\to 1_T\) such that \(\gamma\otimes e:e\otimes e\to e\) is an isomorphism.
Dually, one defines a right idempotent.
For a left idempotent \(e\to 1_T\), the essential image of \(e\otimes-:T\to T\) is a smashing ideal, and the cone of \(e\to 1_T\) gives a right idempotent.
Conversely, for a smashing ideal \(I\), the counit \(\iota_I\iota_I^\mR(1_T)\to 1_T\) is a left idempotent.
These correspondences give bijections between smashing ideals, left idempotents, and right idempotents; see \cite{BF11}.

\begin{defn}
  Let \(T\) be a tt-category.
  A localizing ideal \(I\) is called strongly compactly generated if it is compactly generated as a localizing subcategory.
\end{defn}

\begin{lem}
  Let \(T\) be a tt-category.
  There exists a unique maximal strongly compactly generated localizing ideal \(I\).
\end{lem}
\begin{proof}
  Let \(\{I_\alpha\}_\alpha\) be the collection of all strongly compactly generated localizing ideals of \(T\).
  Since each \(I_\alpha\) is generated by a set of compact objects, and since \(T^\omega\) is essentially small, these ideals form a set.
  For each \(\alpha\), choose a set \(S_\alpha\subset T^\omega\) such that \(I_\alpha=\ang{S_\alpha}\), and set \(S=\bigcup_\alpha S_\alpha\).
  Let \(I=\ang{S}\).
  Since \(I=\ang{\bigcup_\alpha I_\alpha}\) and \(\bigcup_\alpha I_\alpha\) is closed under tensoring with arbitrary objects of \(T\), it follows that \(I\) is a localizing ideal.
  By construction, \(I\) is the maximal strongly compactly generated localizing ideal.
\end{proof}

\begin{defn}\label{defn:maximax_compactly_generated_localizing_ideal}
For a tt-category \(T\), we denote the maximal strongly compactly generated localizing ideal of \(T\) by \(T^{\mcg}\).
\end{defn}

The ideal \(T^{\mcg}\) is the main object of study in this paper.
We will prove that the maximal strongly compactly generated localizing ideal of the category of nuclear modules is equivalent to the derived category of torsion modules.

We record the following lemma for later use.

\begin{lem}\label{lem:mcg_is_smashing_ideal}
  Let \(T\) be a tt-category such that \(T^\omega\otimes T^\omega\subset T^\omega\), \(1_T\in T^\omega\), and \(T^\omega\) is rigid, that is, every object of \(T^\omega\) is dualizable.
  Then any strongly compactly generated localizing ideal of \(T\) is a smashing ideal.
\end{lem}
\begin{proof}
  When \(T=\ang{T^\omega}\), this follows from \cite{Miller1992FiniteLocalizations}.
  In general, let \(I\subset T\) be a strongly compactly generated localizing ideal, and set \(T'=\ang{T^\omega}\).
  Then \(T'\) is a tt-category, \(I\subset T'\) and \(I\) is a smashing ideal in \(T'\).
  Let \(\gamma:e\to 1_{T'}=1_T\) be the left idempotent corresponding to \(I\).
  Since \(I\) is a localizing ideal of \(T\), the essential image of \(e\otimes-:T\to T\) is \(I\).
  Hence \(I\) is a smashing ideal in \(T\).
\end{proof}

\subsection{Categories associated with formal schemes}
In this section, we introduce the derived category of torsion modules, Efimov's category of nuclear modules, and Clausen--Scholze category of nuclear modules over a formal scheme.

Throughout this paper, all formal schemes are assumed to be quasi-compact and quasi-separated, and moreover to locally admit a finitely generated ideal of definition.

\subsubsection{Derived category of torsion modules}

Let \(R\) be a commutative ring, and let \(I\subset R\) be an ideal.
An \(R\)-module \(M\) is called an \(I^\infty\)-torsion module if, for every \(m\in M\), there exists \(n\ge 1\) such that \(I^nm=0\).
Let \(D(R)\) be the derived \(\infty\)-category of \(R\)-modules.
A complex \(F\in D(R)\) is called \(I^\infty\)-torsion if its cohomologies are \(I^\infty\)-torsion.
Let \(\tors{I}(R)\) be the full subcategory of \(D(R)\) spanned by the \(I^\infty\)-torsion objects.
If \(I=(a_1,\ldots,a_n)\) is finitely generated, then \(\tors{I}(R)\) is a smashing ideal of \(D(R)\), and it is generated by the single compact object \(\Kos(R;a_1,\ldots,a_n)\).
The quotient \(D(R)/\tors{I}(R)\) is equivalent to \(D(\Spec R\setminus V(I))\), where \(V(I)\) is the closed subset of \(\Spec R\) corresponding to \(I\).
We denote the classical \(I\)-adic completion of \(R\) by \(R_I^{\wedge}\).

For an affine formal scheme \(U=\Spf(R_I^\wedge)\), we set \(\Dtors(U)=\tors{I}(R)\).
This construction does not depend on the choice of \(R\) and satisfies Zariski descent.
For a general formal scheme \(\Frak{X}\), we define \(\Dtors(\Frak{X})\) by Zariski descent.
Then \(\Dtors(\Frak{X})\) is finitely compactly generated as a localizing subcategory of itself.
For an open subset \(U\subset \Frak{X}\), the restriction functor \(\Dtors(\Frak{X})\to \Dtors(U)\) is a quotient functor whose kernel is a compactly generated smashing ideal.

\subsubsection{Category of nuclear modules of a tt-\(\infty\)-category}

Let \(\pr\) be the \(\infty\)-category of presentable stable \(\infty\)-categories and continuous, that is, colimit-preserving, functors.
A commutative algebra object in \(\pr\) is called a tt-\(\infty\)-category.
A homomorphism between tt-\(\infty\)-categories is a morphism in \(\CAlg(\pr)\), that is, a symmetric monoidal continuous functor.
The symmetric monoidal structure on a tt-\(\infty\)-category is automatically closed.
We denote the right adjoint of \(-\otimes x\) by \(\RHom(x,-)\).

Let \(\prcpt\) be the \(\infty\)-category of compactly generated presentable stable \(\infty\)-categories and continuous functors which preserve compact objects.
A commutative algebra object in \(\prcpt\) is called a compactly generated tt-\(\infty\)-category.
A homomorphism between compactly generated tt-\(\infty\)-categories is a morphism in \(\CAlg(\prcpt)\), that is, a symmetric monoidal continuous functor which preserves compact objects.

\begin{defn}[{\cite[Definitions 8.5]{CS26}}]
  Let \(C\) be a compactly generated tt-\(\infty\)-category.
  An object \(x\in C\) is called nuclear if, for every compact object \(c\in C\), the natural map \( \RHom(c,1_C)\otimes x\to \RHom(c,x)\) is an isomorphism.
  The full subcategory of nuclear objects in \(C\) is denoted by \(\Nuc(C)\).
\end{defn}
This subcategory \(\Nuc(C)\) is closed under tensor products and colimits, and is a presentable stable \(\infty\)-category.
Every homomorphism between compactly generated tt-\(\infty\)-categories preserves nuclear objects.
See \cite[Theorem 8.6]{CS26} and \cite[Theorem 2.4]{MeyerWagnerQHodgeRefinedTC}.

\subsubsection{Efimov's category of nuclear modules}

\begin{defn}
  Let \(E\) be a locally rigid tt-\(\infty\)-category (\cite[Definition C.1.1]{AGKRRV20}) whose unit object \(1_E\) is \(\omega_1\)-compact.
  Define \(E^{\rig}=\Nuc(\Ind(E^{\omega_1}))\).
  \end{defn}
  This definition is equivalent to the usual rigidification; see \cite[Theorem 4.2]{Efimov2025LocalizingInvariantsInverseLimits}.

\begin{defn}[{see \cite{Efimov2025LocalizingInvariantsInverseLimits}}]\label{defn:eNuc}
  For a formal scheme \(\Frak{X}\), define \(\eNuc(\Frak{X})=\Dtors(\Frak{X})^\rig\).
\end{defn}

For affine formal schemes, this definition agrees with Efimov's definition.
By the following proposition, it satisfies Zariski descent.
Its proof is given in the appendix.

\begin{prop}\label{prop:descent_for_eNuc}
  For formal schemes, the assignment \(\Frak{X}\mapsto \eNuc(\Frak{X})\) satisfies Zariski descent in \(\pr\).
  Moreover, for every open subset \(U\subset \Frak{X}\), the restriction functor \(\eNuc(\Frak{X})\to \eNuc(U)\) is a quotient functor and its kernel is a smashing ideal.
\end{prop}
\begin{proof}
  This follows from Corollary \ref{cor:rigidification_preserves_smashing_localizations} and Proposition \ref{prop:descent_for_rigidification}.
\end{proof}

\subsubsection{Clausen--Scholze category of nuclear modules}

In this section, we introduce the Clausen--Scholze category of nuclear modules of a formal scheme; see \cite{Scholze2026LecturesAnalyticGeometry,And23,AM24}.

Let \(R\) be a ring, and let \(I\subset R\) be an ideal.
Let \(\complproj\) be the following symmetric monoidal category.
Its objects are \(\prod_A R_I^\wedge\), where \(A\) is a set of at most countable cardinality.
The morphisms are given by
\[
\Hom(\prod_A R_I^\wedge,\prod_B R_I^\wedge)\cong \prod_B\lim_{n\in\nat}(\bigoplus_{A}R/I^n).
\]
The tensor product is given by
\[
(\prod_A R_I^\wedge)\otimes(\prod_B R_I^\wedge)\cong \prod_{A\times B}R_I^\wedge.
\]

Let \(\complproj\text{-}\mathrm{Mod}\) be the classical category of \(\complproj\)-modules, or equivalently, the category of \(\End_{\complproj}(\prod_{\nat}R_I^\wedge)\)-modules.

\begin{defn}[{\cite[Proposition 7.5]{Efimov2025LocalizingInvariantsInverseLimits}}]
  For a Noetherian affine formal scheme \(U=\Spf(R_I^\wedge)\), define \(\csNuc(U)=\Nuc(D(\complproj\text{-}\mathrm{Mod}))\).
  For a Noetherian formal scheme \(\Frak{X}\), define \(\csNuc(\Frak{X})\) by Zariski descent.
\end{defn}

For an affine formal scheme \(U=\Spf(R_I^\wedge)\), there are natural equivalences \(\Perf(R_I^\wedge)\cong \csNuc(U)^\omega \cong \eNuc(U)^\omega\).

\section{Maximal strongly compactly generated localizing ideals of nuclear modules}

The purpose of this section is to prove Theorem \ref{thm:mcg_is_torsions}, which asserts that \(\csNuc(\Frak{X})^\mcg\cong \eNuc(\Frak{X})^\mcg\cong\Dtors(\Frak{X})\).
We begin with two lemmas.
Lemma \ref{lem:torsion_is_localizing_ideal} shows that \(\Dtors(U)\) is a localizing ideal of the category of nuclear modules, and Lemma \ref{lem:torsion_is_maximal_cg_localizing_ideal} proves its maximality.

In this section, all formal schemes are assumed to be Noetherian.

For an affine formal scheme \(U=\Spf(R_I^\wedge)\), the equivalence \(\Perf(R_I^\wedge)\cong \csNuc(U)^\omega\) induces a fully faithful strongly continuous symmetric monoidal functor
\[
\rho:D(R_I^\wedge)\cong \Ind(\Perf(R_I^\wedge))\to \csNuc(U).
\]
By composing it with the inclusion \(\tors{I}(R)\subset D(R_I^\wedge)\), we obtain a functor \(\Dtors(U)\to \csNuc(U)\).
By patching these functors, we obtain
\[
\Dtors(\Frak{X})\to \csNuc(\Frak{X})
\]
for any formal scheme \(\Frak{X}\).
Via this functor, we identify \(\Dtors(\Frak{X})\) with a localizing subcategory of \(\csNuc(\Frak{X})\).

By patching the functor of \cite[Corollary 7.6]{Efimov2025LocalizingInvariantsInverseLimits}, there is a fully faithful strongly continuous symmetric monoidal functor
\[
\csNuc(\Frak{X})\to \eNuc(\Frak{X}).
\]
We identify \(\csNuc(\Frak{X})\) with a full subcategory of \(\eNuc(\Frak{X})\).

\begin{lem}\label{lem:torsion_is_localizing_ideal}
  Let \(U=\Spf(R_I^\wedge)\) be a Noetherian affine formal scheme, let \(F\in \tors{I}(R_I^\wedge)\), and let \(x\in \eNuc(U)\).
  Then \(\rho(F)\otimes x\) lies in the essential image of \(\tors{I}(R_I^\wedge)\).
\end{lem}
\begin{proof}
  The localizing subcategory \(\tors{I}(R_I^\wedge)\subset D(R_I^\wedge)\) is generated by \(R_I^\wedge/I\).
  Thus we may assume that \(F=R_I^\wedge/I\).
  The inclusion \(\eNuc(U)\subset \Ind(\tors{I}(R)^{\omega_1})\) induces a fully faithful strongly continuous symmetric monoidal functor
  \begin{align*}
    \eNuc(U)\otimes_{D(R_I^\wedge)}D(R_I^\wedge/I)
    &\to \Ind(\tors{I}(R)^{\omega_1})\otimes_{D(R_I^\wedge)}D(R_I^\wedge/I)\\
    &\cong \Ind(\tors{I}(R)^{\omega_1}\otimes_{\Perf(R_I^\wedge)}\Perf(R_I^\wedge/I))
    \subset \Ind(D(R_I^\wedge/I)^{\omega_1}).
  \end{align*}
  Recall that nuclear modules are preserved by any homomorphism between compactly generated tt-\(\infty\)-categories.
  Since \(\Nuc(D(R_I^\wedge/I))\cong D(R_I^\wedge/I)\), the above functor factors through \(\Nuc(\Ind(D(R_I^\wedge/I)^{\omega_1}))=\eNuc(R_I^\wedge/I)=D(R_I^\wedge/I)\).
  This proves the assertion.
\end{proof}

\begin{lem}\label{lem:torsion_is_maximal_cg_localizing_ideal}
  Let \(U=\Spf(R_I^\wedge)\) be a Noetherian affine formal scheme, and let \(F\in D(R_I^\wedge)\).
  Assume that, for every \(x\in \csNuc(U)\), the object \(\rho(F)\otimes x\) lies in the essential image of \(\rho\).
  Then \(F\in \tors{I}(R_I^\wedge)\).
\end{lem}
\begin{proof}
  Assume that \(F\notin \tors{I}(R_I^\wedge)\).
The category \(\csNuc(U)\), regarded as a localizing subcategory of \(D(\complproj\text{-}\mathrm{Mod})\), is generated by \(G\coloneqq \RHom_{D(\complproj\text{-}\mathrm{Mod})}\left(\prod_{\nat} R_I^\wedge, R_I^\wedge\right)\) as a localizing subcategory \cite[Proposition 7.4]{Efimov2025LocalizingInvariantsInverseLimits}.
Hence, it suffices to prove that \(\rho(F)\otimes G\) does not lie in the essential image of \(\rho\).

For \(x\in \csNuc(U)\), the mapping spectrum \(\sHom(R_I^\wedge,x)\) is naturally an \(\End(R_I^\wedge)^\op=R_I^\wedge\)-module.
This construction defines a continuous functor \(\sigma\colon \csNuc(U)\to D(R_I^\wedge)\).
This functor is right adjoint to \(\rho\).
Thus, it suffices to prove that the counit map
\(\varepsilon\colon \rho\sigma(\rho(F)\otimes G)\to \rho(F)\otimes G\)
is not an equivalence.
We will compute \(\sigma\RHom\left(\prod_{\nat} R_I^\wedge,\varepsilon\right)\).

  Note that, for \(H\in D(R_I^\wedge)\), \(x\in \csNuc(U)\), and any set \(A\), there is an equivalence
  \[
  \sigma\RHom(\prod_A R_I^\wedge, \rho(H)\otimes x)\cong H\otimes \sigma\RHom(\prod_{A\times \nat}R_I^\wedge, x).
  \]
  This is obvious for \(H=R_I^\wedge\).
  Since \(\prod_A R_I^\wedge\) is compact, both sides preserve colimits in \(H\).
  Hence the equivalence holds for every \(H\in \ang{R_I^\wedge}=D(R_I^\wedge)\).

  Therefore,
  \begin{align*}
    \sigma\RHom(\prod_A R_I^\wedge, \rho(F)\otimes G)
    \cong F\otimes \sigma\RHom(\prod_{A\times \nat}R_I^\wedge, G)
    \cong F\otimes \lim_{n\in \nat}(\bigoplus_{A\times \nat} R/I^n).
  \end{align*}
  \begin{align*}
    \sigma\RHom(\prod_\nat R_I^\wedge,\rho \sigma(\rho(F)\otimes G))
    &\cong F\otimes \sigma\RHom(\prod_\nat R_I^\wedge,\rho \sigma(G))\\
    &\cong F\otimes \sigma\RHom(\prod_\nat R_I^\wedge,\rho(\lim_{n\in \nat}(\bigoplus_{\nat} R/I^n)))\\
    &\cong F\otimes (\lim_{n\in \nat}(\bigoplus_{\nat} R/I^n))\otimes \sigma\RHom(\prod_\nat R_I^\wedge,R_I^\wedge)\\
    &\cong F\otimes (\lim_{n\in \nat}(\bigoplus_{\nat} R/I^n))\otimes (\lim_{n\in \nat}(\bigoplus_{\nat} R/I^n)).
  \end{align*}
  Thus \(\sigma\RHom(\prod_\nat R_I^\wedge,\varepsilon)\) is the map
  \[
  F\otimes (\lim_{n\in \nat}(\bigoplus_{\nat} R/I^n))\otimes (\lim_{n\in \nat}(\bigoplus_{\nat} R/I^n))\to F\otimes \lim_{n\in \nat}(\bigoplus_{\nat\times \nat} R/I^n).
  \]
  Since \(F\notin \tors{I}(R_I^\wedge)\), there exists a prime ideal \(\Frak{p}\subset R_I^\wedge\) such that \(I\not\subset \Frak{p}\) and \(F\otimes \kappa(\Frak{p})\neq0\), where \(\kappa(\Frak{p})\) is the residue field.
  We may replace \(F\) by \(F\otimes \kappa(\Frak{p})\).
  Since \(F\otimes \kappa(\Frak{p})\) is a direct sum of shifts of \(\kappa(\Frak{p})\), we may further replace it by \(\kappa(\Frak{p})\).
  Then both sides are concentrated in degree \(0\) (see, for example, \cite[Tag 06LE]{stacks-project}), and they are identified with submodules of \(\prod_{\nat\times \nat}\kappa(\Frak{p})\).
  Choose \(f\in I\setminus \Frak{p}\).
  The diagonal matrix \((f,f^2,f^3,\ldots)\) in \(\prod_{\nat\times \nat}\kappa(\Frak{p})\) lies in \(\kappa(\Frak{p})\otimes \lim_{n\in \nat}(\bigoplus_{\nat\times \nat} R/I^n).\)
  On the other hand, it does not lie in \(\kappa(\Frak{p})\otimes(\lim_{n\in \nat}(\bigoplus_{\nat} R/I^n))\otimes (\lim_{n\in \nat}(\bigoplus_{\nat} R/I^n)),\) since it has infinite rank.
  Thus \(\varepsilon\) is not an equivalence.
\end{proof}

\begin{thm}\label{thm:mcg_is_torsions}
  Let \(\Frak{X}\) be a Noetherian formal scheme.
  Then \(\eNuc(\Frak{X})^\mcg\cong \Dtors(\Frak{X})\).
  Moreover, \(\Dtors(\Frak{X})\) is a strongly compactly generated localizing ideal of \(\csNuc(\Frak{X})\).
  If there exists an affine open covering \(\{U_i\}_i\) of \(\Frak{X}\) such that each restriction functor \(\csNuc(\Frak{X})\to \csNuc(U_i)\) is a quotient functor, then \(\csNuc(\Frak{X})^\mcg\cong \Dtors(\Frak{X})\).
\end{thm}
\begin{proof}
  By Lemma \ref{lem:torsion_is_localizing_ideal}, \(\Dtors(\Frak{X})\) is a strongly compactly generated localizing ideal of \(\eNuc(\Frak{X})\).
  Hence it is also a strongly compactly generated localizing ideal of \(\csNuc(\Frak{X})\).
  We will prove that it is maximal.
  Let \(K\subset \eNuc(\Frak{X})\) be a strongly compactly generated localizing ideal.
  For an open subset \(U\subset\Frak{X}\), let \(K_U\) be the localizing subcategory generated by the image of \(K\) in \(\eNuc(U)\).
  Then \(K_U\) is compactly generated as a localizing subcategory.
  Since \(\eNuc(\Frak{X})\to\eNuc(U)\) is a quotient functor, \(K_U\) is automatically a localizing ideal.
  Suppose, for contradiction, that \(K\not\subset \Dtors(\Frak{X})\).
  Then there exists an affine open subset \(U\cong \Spf(R_I^\wedge)\subset \Frak{X}\) such that \(K_U\not\subset\Dtors(U)\).
  Since \(K_U\) is strongly compactly generated, \(K_U\) lies in the essential image of \(\rho:D(R_I^\wedge)\to \eNuc(U)\).
  Since \(\csNuc(U)\subset \eNuc(U)\), Lemma \ref{lem:torsion_is_maximal_cg_localizing_ideal} implies that \(K_U\subset \Dtors(U)\), a contradiction.
  Thus \(\eNuc(\Frak{X})^\mcg\cong \Dtors(\Frak{X})\).

  The proof for \(\csNuc(\Frak{X})\) is the same.
\end{proof}

We will use the following lemma later.

\begin{lem}\label{lem:quotient_of_nuclear_modules}
  Let \(U=\Spf(A_I^\wedge)\) be a Noetherian affine formal scheme, and let \(A\to B\) be a homomorphism of rings which factors through \(A/I^m\) for some \(m\ge1\).
  Then
  \[
  \csNuc(U)\otimes_{D(A_I^\wedge)}D(B)\cong \eNuc(U)\otimes_{D(A_I^\wedge)}D(B)\cong D(B).
  \]
\end{lem}
\begin{proof}
  Since \(\tors{I}(A)\) is a smashing ideal of \(\eNuc(U)\), we have
  \[
  \eNuc(U)\otimes_{D(A_I^\wedge)}D(B)\cong \eNuc(U)\otimes_{D(A_I^\wedge)}\tors{I}(A)\otimes_{\tors{I}(A)}D(B)\cong \tors{I}(A)\otimes_{\tors{I}(A)}D(B)\cong D(B).
  \]
  The same argument applies to \(\csNuc(U)\).
\end{proof}

\section{Balmer spectrum of a weakly compactly generated tensor triangulated category}

In this section, we introduce Balmer's tensor triangular spectrum of a tt-category, as defined in \cite{BalmerSpectrumPrimeIdeals}.
We will apply this construction to the derived category of torsion modules on a formal scheme in Theorem \ref{thm:Spec_of_Dtors}.
However, the unit object of this category is not compact, so we need a slight modification.
Roughly speaking, our spectrum gives a spectrum of localizing ideals of a compactly generated tt-category.

\begin{defn}
  We say that a tt-category \(T\) is weakly compactly generated if it is compactly generated as a localizing subcategory of itself and \(T^\omega\otimes T^\omega\subset T^\omega\).
  If, moreover, \(1_T\) is compact, then \(T\) is called a compactly generated tt-category.
\end{defn}
Every weakly compactly generated localizing ideal of a weakly compactly generated tt-category is strongly compactly generated.
Moreover, compactly generated localizing ideals are in bijection with thick tensor ideals of \(T^\omega\), that is, thick subcategories \(I\subset T^\omega\) such that, for any \(a\in T^\omega\) and \(b\in I\), one has \(a\otimes b\in I\).

In this section, we assume that \(T\) is a weakly compactly generated tt-category.

\begin{defn}
  Let \(I\subset T\) be a compactly generated localizing ideal.
  We say that \(I\) is prime if \(I\neq T\) and \(a\otimes b\notin I\) for all \(a,b\in T^\omega\setminus I^\omega\).
\end{defn}

Let \(\Spc(T)\) be the set of prime ideals.
We define a topology on this set as follows.
The support of a subset \(S\subset T\) is defined by
\[
\Supp(S)=\{P\in\Spc(T)\mid S\not\subset P\}.
\]
The closed subsets of \(\Spc(T)\) are generated by \(\Supp(a)\) for \(a\in T^\omega\).

\begin{defn}\label{defn:presheaf_of_triangulated_categories}
For a subset \(U\subset \Spc(T)\), we set
\[
I_U=\ang{a\in T^\omega\mid \Supp(a)\cap U=\emptyset},\quad T_U=\Idem(T/I_U).
\]
Here \(\Idem\) denotes idempotent completion.
\end{defn}

The subcategory \(I_U\) is a compactly generated localizing ideal, and hence \(T_U\) has a symmetric monoidal structure.
Let \(\sO_T\) be the sheaf of commutative rings on \(\Spc(T)\) associated with the presheaf
\[
U\mapsto \End(1_{T_U}).
\]

With our formalism, we can rephrase the definition of the Balmer spectrum in \cite{BalmerSpectrumPrimeIdeals} as follows.

\begin{defn}
The ringed space
\[
\Spec(T)=(\Spc(T),\sO_T)
\]
is called the Balmer spectrum of a weakly compactly generated tt-category \(T\).
\end{defn}
In fact, it is a locally ringed space.
The proof is the same as that for the usual Balmer spectrum, so we omit it.

\begin{rem}\label{rem:spcetrum_of_compactly_generated_tt-category}
If \(T\) is a compactly generated tt-category, then our spectrum coincides with the usual Balmer spectrum of \(T^\omega\).
\end{rem}

This definition satisfies the following functoriality.

\begin{defn}
Let \(T,T'\) be weakly compactly generated tt-categories.
A homomorphism \(F:T\to T'\) of tt-categories is called cg-preserving if, for every \(a\in T^\omega\), the localizing ideal \(\oang{F(a)}\) is weakly finitely compactly generated.
\end{defn}

Every compact-object-preserving homomorphism is cg-preserving.
In particular, every equivalence is cg-preserving.

Let \(\ideals\) be the \((2,1)\)-category of weakly compactly generated tt-categories, cg-preserving homomorphisms, and natural isomorphisms.
Let \(\LocSpaces\) be the category of locally ringed spaces.

\begin{prop}
  Let \(F:T\to T'\) be a cg-preserving homomorphism of weakly compactly generated tt-categories.
  Then there is a natural morphism of ringed spaces
  \[
    \Spec(T')\to \Spec(T):Q\mapsto\ang{\{a\in T^\omega\mid F(a)\in Q\}}.
  \]
  Moreover, this construction defines a contravariant functor
  \[
    \Spec:(\ideals)^\op\to \LocSpaces.
  \]
\end{prop}
\begin{proof}
  Let \(Q\in\Spc(T')\), and put \(P=\ang{\{a\in T^\omega\mid F(a)\in Q\}}\).
  We first show that \(P\) is prime.
  Recall that \(P\cap T^\omega=\{a\in T^\omega\mid F(a)\in Q\}\).
  For \(a\in T^\omega\), let \(g(a)\) be a compact generator of \(\oang{F(a)}\).
  Then \(F(a)\in Q\) if and only if \(g(a)\in Q\).
  If \(a,b\in T^\omega\) and \(a,b\notin P\), then \(g(a),g(b)\notin Q\), and hence \(g(a)\otimes g(b)\notin Q\).
  Since \(g(a)\otimes g(b)\) is a compact generator of \(\oang{F(a\otimes b)}\), it follows that \(a\otimes b\notin P\), as desired.
  Next, we prove that \(P\neq T\).
  If \(P=T\), then \(Q\) contains \(1_{T'}\cong F(1_T)\in \ang{F(T^\omega)}\), and hence \(Q=T'\), a contradiction.
  Thus \(P\neq T\), and \(P\) is prime.
  For \(a\in T^\omega\), the inverse image of \(\Supp(a)\) is \(\Supp(F(a))=\Supp(g(a))\).
  Hence this map is continuous.

  We now construct a morphism between the structure sheaves.
  Let \(U\subset \Spc(T)\) be an open subset, and let \(V\) be its inverse image in \(\Spc(T')\).
  For \(a\in T^\omega\) with \(\Supp(a)\cap U=\emptyset\), we have \(\Supp(F(a))\cap V=\emptyset\).
  Hence \(F(a)\) is zero in \(T'_V\).
  This shows that the composition \(T\to T'\to T'_V\) factors through \(T_U\).
  The induced functor \(T_U\to T'_V\) is automatically symmetric monoidal, and the induced homomorphism \(\End(1_{T_U})\to \End(1_{T'_V})\) gives the desired morphism.
  Functoriality follows from the construction.
\end{proof}

We now give two basic tools for computing the spectrum.
Proposition \ref{prop:ideals} is analogous to closed embeddings, while Example \ref{example:quotient} is analogous to open embeddings in ordinary algebraic geometry.

\begin{prop}\label{prop:ideals}
Let \(T\) be a weakly compactly generated tt-category, and let \(I\subset T\) be a weakly finitely compactly generated smashing ideal.
Then \(I\) is a weakly compactly generated tt-category, and the right adjoint \(\iota_I^\mR:T\to I\) of the inclusion is a cg-preserving homomorphism.
The map \(\Spc(\iota_I^{\mR}):\Spc(I)\to \Spc(T)\) is a closed embedding whose image is \(\Supp(d)\), where \(d\in I\) is a compact generator.
\end{prop}
\begin{proof}
  Clearly, \(I\) is a weakly compactly generated tt-category.
  Let \(d\in I\) be a compact generator of \(I\) as a localizing ideal of \(T\).
  For any \(a\in T^\omega\), we have \(a\otimes d\in I\cap T^\omega=I^\omega\), and \(\oang{\iota_I^{\mR}(a)}=\oang{a\otimes d}.\)
  Hence \(\iota_I^{\mR}\) is cg-preserving.
  It is clear that the image of \(\Spc(\iota_I^{\mR})\) is contained in \(\Supp(d)\).

  We now construct the inverse map.
  Let \(P\subset T\) be a prime ideal such that \(d\notin P\).
  Then \(\ang{\iota_I^{-1}(P^\omega)}\) is a prime ideal of \(I\).
  This construction defines a map \(\phi\colon \Supp(d)\to \Spc(I).\)
  We claim that \(\phi\) is continuous and inverse to \(\Spc(\iota_I^{\mR})\).
  The equality \(\phi\circ\Spc(\iota_I^{\mR})=\id_{\Spc(I)}\) is immediate.

  To prove \(\Spc(\iota_I^{\mR})\circ\phi=\id_{\Supp(d)}\), it suffices to show that, for any prime ideal \(P\subset T\) with \(d\notin P\) and any \(a\in T^\omega\), we have
\[
  a\in P
  \quad\Longleftrightarrow\quad
  \iota_I^{\mR}(a)\in \iota_I^{-1}(P).
\]
The latter condition is equivalent to \(a\otimes d\in P\cap I\).
The implication \(\Rightarrow\) is immediate.
Conversely, if \(a\otimes d\in P\cap I\), then \(a\in P\), since \(d\notin P\) and \(P\) is prime.

  Finally, for any \(a\in T^\omega\), we have
  \[
    \Supp_{\Spc(I)}(a\otimes d)=\phi^{-1}(\Supp_{\Spc(T)}(a)).
  \]
  Hence \(\phi\) is continuous.
\end{proof}

\begin{example}[{cf. \cite[Proposition 3.11]{BalmerSpectrumPrimeIdeals}}]\label{example:quotient}
  Let \(T\) be a weakly compactly generated tt-category, and let \(I\subset T\) be a compactly generated localizing ideal.
  Then the quotient \(T/I\) is also a weakly compactly generated tt-category, and the quotient functor \(L_I:T\to T/I\) is a cg-preserving homomorphism.
  The morphism \(\Spec(L_I):\Spec(T/I)\to \Spec(T)\) is an embedding whose image is \(\Spc(T)\setminus \bigcup_{a\in I^\omega}\Supp(a)\).
\end{example}

For a formal scheme \(\Frak{X}\), \(\Dtors(\Frak{X})\) is a weakly compactly generated tt-category.

\begin{thm}\label{thm:Spec_of_Dtors}
  Let \(\Frak{X}\) be a quasi-compact quasi-separated formal scheme.
  Assume that there exists an affine open cover \(\{U_i\}\) such that each \(U_i\) is isomorphic to \(\Spf(R_I^\wedge)\) for some ring \(R\) and some ideal \(I\) generated by a weakly proregular sequence (e.g., a Noetherian formal scheme).
  Then there is a natural isomorphism
  \[
  \Spec(\Ho(\Dtors(\Frak{X})))\cong \Frak{X}
  \]
  of ringed spaces.
\end{thm}
\begin{proof}
  By Example \ref{example:quotient}, it suffices to prove the assertion when \(\Frak{X}\) is affine, say \(\Frak{X}=\Spf(R_I^\wedge)\).
  By \cite[Theorem 4.1]{Balmer2019GuideTTClassification} and Remark \ref{rem:spcetrum_of_compactly_generated_tt-category}, there is a homeomorphism \(\Spc(D(R))\cong \Spec R\).
  Since \(\tors{I}(R)\) is a compactly generated smashing ideal of \(D(R)\), Proposition \ref{prop:ideals} gives a homeomorphism of underlying spaces \(\Spc(\tors{I}(R))\cong V(I)\subset \Spec R\).
  It remains to identify the structure sheaf.
  The ring \(\End(1_{\tors{I}(R)})\) is the derived \(I\)-completion of \(R\), which is isomorphic to \(R_I^\wedge\) by the weak proregularity assumption.
  This gives an isomorphism of structure sheaves.
\end{proof}

\section{Reconstruction of formal schemes}

In this section, we prove our main result, Theorem \ref{thm:reconstruction}.
We first define some categories in order to state the theorem.

Let \(\largett\) be the \((2,1)\)-category of tt-categories \(T\) such that \(T^\omega\otimes T^\omega\subset T^\omega\), \(1_T\in T^\omega\), and \(T^\omega\) is rigid, that is, every object of \(T^\omega\) is dualizable.
In this case, \(T^\mcg\) is a smashing ideal by Lemma \ref{lem:mcg_is_smashing_ideal}.
Morphisms in \(\largett\) are compact-object-preserving homomorphisms \(F:T\to T'\) such that \(\iota^{\mR}_{(T')^\mcg}\circ F:T\to T'\to (T')^\mcg\) factors through \(\iota^{\mR}_{T^\mcg}:T\to T^\mcg\), and such that the induced functor \(F^\mcg:T^\mcg\to (T')^\mcg\) is cg-preserving.
By the same argument as in Proposition \ref{prop:ideals}, if \((T')^\mcg\subset T'\) is weakly finitely compactly generated, then the last condition is automatically satisfied.
The \(2\)-morphisms in \(\largett\) are natural isomorphisms.

The correspondence \(T\mapsto T^\mcg\), \((F:T\to T')\mapsto (F^\mcg:T^\mcg\to (T')^\mcg)\), defines a functor
\[
  (-)^\mcg:\largett\to \ideals.
\]

Let \(\higherlargett\) be the subcategory of \(\CAlg(\dualpr)\) whose objects are those \(C\) such that every compact object of \(C\) is dualizable.
Its morphisms are strongly continuous symmetric monoidal functors \(F:C\to D\) such that \(C\to D\to D^\mcg\) factors through \(C^\mcg\), and such that the induced functor \(F^\mcg:C^\mcg\to D^\mcg\) is cg-preserving.

Let \(\Formalsch\) be the (classical) category of Noetherian formal schemes.
Let \(\nFormalsch\) be the full subcategory of \(\Formalsch\) whose objects are Noetherian formal schemes \(\Frak{X}\) such that there exists an affine open covering \(\{U_i\}_i\) for which each restriction functor \(\csNuc(\Frak{X})\to \csNuc(U_i)\) is a quotient functor.
Every Noetherian affine formal scheme belongs to \(\nFormalsch\).

\begin{thm}\label{thm:reconstruction}
  The compositions of the following functors are the canonical inclusions:
  \begin{align*}
    \Formalsch^\op&\os{\eNuc}{\to} \higherlargett\os{\Ho}{\to} \largett\os{(-)^\mcg}{\to} \ideals\os{\Spec}{\to} {\LocSpaces}^\op,\\
    (\nFormalsch)^\op&\os{\csNuc}{\to} \higherlargett\os{\Ho}{\to} \largett\os{(-)^\mcg}{\to} \ideals\os{\Spec}{\to} {\LocSpaces}^\op.
  \end{align*}
\end{thm}
\begin{proof}
This follows by combining Theorem \ref{thm:mcg_is_torsions} and Theorem \ref{thm:Spec_of_Dtors}.
\end{proof}

\begin{cor}\label{cor:identical_formal_schemes}
  \begin{enumerate}
  \item Let \(\Frak{X},\Frak{Y}\) be Noetherian formal schemes.
  If there is a symmetric monoidal triangulated equivalence \(\Ho(\eNuc(\Frak{X}))\cong \Ho(\eNuc(\Frak{Y}))\), then \(\Frak{X}\cong \Frak{Y}\).
  \item Let \(\Frak{X},\Frak{Y}\in \nFormalsch\).
  If there is a symmetric monoidal triangulated equivalence \(\Ho(\csNuc(\Frak{X}))\cong \Ho(\csNuc(\Frak{Y}))\), then \(\Frak{X}\cong \Frak{Y}\).
  \end{enumerate}
\end{cor}
\begin{proof}
  The category \(\largett\) is closed under equivalences in the \((2,1)\)-category of tt-categories.
\end{proof}

\begin{cor}\label{cor:fully_faithful_weak_ver}
  \begin{enumerate}
  \item  The functor
  \[
    \eNuc:\Formalsch^\op\to \higherlargett
  \]
  is fully faithful.
  
  \item For any \(\Frak{X}\in \nFormalsch\) and \(\Frak{Y}\in\Formalsch\), there is a natural equivalence
  \[
   \Map_{\higherlargett}(\csNuc(\Frak{X}),\eNuc(\Frak{Y}))\cong \Hom(\Frak{Y},\Frak{X}).
  \]

  \item The functor
  \[
    \csNuc:(\nFormalsch)^\op\to \higherlargett
  \]
  is fully faithful.
\end{enumerate}
\end{cor}
\begin{proof}
  (1)
  Let \(\Frak{X},\Frak{Y}\) be Noetherian formal schemes, and let \(F:\eNuc(\Frak{X})\to \eNuc(\Frak{Y})\) be a morphism in \(\higherlargett\).
  Put \(f=\Spec(F^\mcg):\Frak{Y}\to\Frak{X}\).
  Let \(U=\Spf(A_I^\wedge)\subset \Frak{X}\) and \(V=\Spf(B_J^\wedge)\subset \Frak{Y}\) be affine open subsets such that \(f(V)\subset U\).
  Then \(F^\mcg\) induces
  \[
  \Dtors(U)=(\Dtors(\Frak{X}))_U\to (\Dtors(\Frak{Y}))_V=\Dtors(V),
  \]
  where we use the notation of Definition \ref{defn:presheaf_of_triangulated_categories}.
  By the universal property of rigidification, \(F\) induces a functor
  \[
  \eNuc(U)= \Dtors(U)^\rig\to \Dtors(V)^\rig= \eNuc(V).
  \]
  Therefore, we may assume that \(\Frak{X}\) and \(\Frak{Y}\) are affine.

  There is an equivalence \(\eNuc(\Frak{Y})\cong \lim_{n\in\nat}^{\mathrm{dual}}D(B/J^n)\), where \(\lim^{\mathrm{dual}}\) denotes the limit in \(\dualpr\); see \cite[Proposition 5.23, Lemma 5.26]{Efimov2025LocalizingInvariantsInverseLimits}.
  Let \(K_n\) be the kernel of the induced homomorphism
  \[
  A_I^\wedge=\End(1_{\eNuc(\Frak{X})})\to \End(1_{\eNuc(\Frak{Y})})=B_J^\wedge\to B/J^n.
  \]
  Consider the composition
  \[
  \eNuc(\Frak{X})\to \eNuc(\Frak{Y})\to \Dtors(\Frak{Y})\to D(B/J^n).
  \]
  Since it factors through \(\Dtors(\Frak{X})\), we have \(I^m\subset K_n\) for some \(m\ge1\).
  Hence, by Lemma \ref{lem:quotient_of_nuclear_modules}, this composition uniquely factors through \(\eNuc(\Frak{X})\otimes_{D(A_I^\wedge)}D(A/I^m)\cong D(A/I^m).\)
  Therefore,
  \begin{align*}
  \Map_{\higherlargett}(\eNuc(\Frak{X}),\eNuc(\Frak{Y}))
  &\cong  \lim_{n\in\nat} \colim_{m\in\nat} \Map_{\CAlg(\dualpr)}(D(A/I^m), D(B/J^n))\\
  &\cong  \lim_{n\in\nat} \colim_{m\in\nat} \Hom(A/I^m,B/J^n)\\
  &\cong \Hom(\Frak{Y},\Frak{X}).
  \end{align*}

  (2)
  Let \(F:\csNuc(\Frak{X})\to \eNuc(\Frak{Y})\) be a morphism in \(\higherlargett\).
  Put \(f=\Spec(F^\mcg)\), and let \(U\subset \Frak{X}\) and \(V\subset \Frak{Y}\) be open subsets such that \(f(V)\subset U\).
  Then \(F\) induces
  \[
  \csNuc(U)\to \eNuc(U)=(\Dtors(U))^\rig=((\Dtors(\Frak{X}))_U)^\rig\to ((\Dtors(\Frak{Y}))_V)^\rig=(\Dtors(V))^\rig=\eNuc(V).
  \]
  The rest of the proof is the same as that of (1).

  (3)
  Since \(\csNuc(\Frak{Y})\) is a full subcategory of \(\eNuc(\Frak{Y})\), the map
  \[
  \Map_{\higherlargett}(\csNuc(\Frak{X}),\csNuc(\Frak{Y}))\to \Map_{\higherlargett}(\csNuc(\Frak{X}),\eNuc(\Frak{Y}))
  \]
  is \((-1)\)-truncated.
  Since there is a natural map \(\Hom(\Frak{Y},\Frak{X})\to \Map_{\higherlargett}(\csNuc(\Frak{X}),\csNuc(\Frak{Y}))\), part (2) implies that it is surjective on \(\pi_0\), and hence it is an equivalence.
\end{proof}

\section{Field-valued points and full faithfulness}

In this section, we give another approach to reconstructing formal schemes from categories of nuclear modules, in the spirit of Liu's functorial spectrum \cite{Liu2011FunctorsTriangulatedTensorCategories}.
Symmetric monoidal strongly continuous functors \(\eNuc(\Frak{X})\to D(K)\) to the derived category of a field \(K\) play the role of geometric points; see Theorem \ref{thm:points_reconstruction}.
As an application, we prove the full faithfulness of the functors \(\eNuc\) and \(\csNuc\) in Theorem \ref{thm:fully_faithfulness_strong_ver}.

Let \(R\) be a ring, let \(K\) be a field, and let \(R\to K\) be a homomorphism.
Let \(\Frak{p}\subset R\) be its kernel, and let \(\kappa(\Frak{p})\) be the residue field.
The transcendence degree \(\trdeg_R K\) of \(K\) over \(R\) is defined to be the transcendence degree of \(K\) over \(\kappa(\Frak{p})\).

Let \(\Frak{c}\) be the cardinality of the continuum.

\begin{defn}
  Let \(k\) be a ring, and let \(\Frak{X}\) be a \(k\)-formal scheme.
  We say that \(\Frak{X}\) is nearly of \(<\Frak{c}\)-type if, for every point \(x\in\Frak{X}\), one has \(\trdeg_k(\kappa(x))<\Frak{c}\).
\end{defn}
For example, any formal scheme topologically of finite type over \(k\) is nearly of \(<\Frak{c}\)-type.

Let \(\smFormalsch_k\) be the category of Noetherian formal \(k\)-schemes \(\Frak{X}\) such that there exists an affine open covering \(\{U_i\}_i\) for which each restriction functor \(\csNuc(\Frak{X})\to \csNuc(U_i)\) is a quotient functor whose kernel is a smashing ideal.
Every Noetherian affine formal \(k\)-scheme belongs to \(\smFormalsch_k\).

\begin{lem}\label{lem:trancendental_degree_of_residue_field}
  Let \(k\) be \(\intg\), respectively a field, let \(R\) be a \(k\)-algebra, let \(I\subset R\) be an ideal, and let \(\Frak{p}\subset R_I^\wedge\) be a prime ideal.
  If \(\trdeg_k\kappa(\Frak{p})<\Frak{c}\), then \(I\subset \Frak{p}\).
\end{lem}
\begin{proof}
  Assume that \(I\not\subset \Frak{p}\).
  Take \(f\in I\) such that \(f\notin \Frak{p}\).

  First assume that \(k\) is a field.
  Then there is a homomorphism \(k((t))\to \kappa(\Frak{p})\) sending \(t\) to \(f\).
  Hence \(\trdeg_k \kappa(\Frak{p})\ge \trdeg_k k((t))\ge \Frak{c}\), a contradiction.

  Next assume that \(k=\intg\).
  We claim that the map
  \[
  \{0,1\}^\nat\to R_I^\wedge/\Frak{p}:\{\epsilon_n\}_{n\in\nat}\mapsto \sum_{n}\epsilon_n f^n
  \]
  is injective.
  Take \(\{\epsilon_n\}_n\neq \{\epsilon_n'\}_{n}\in \{0,1\}^\nat\).
  Let \(n_0\) be the smallest integer such that \(\epsilon_{n_0}\neq \epsilon_{n_0}'\).
  Then
  \[
  \sum_n \epsilon_n f^n-\sum_n \epsilon_n' f^n=f^{n_0}(\pm1 +f(\sum_{n>n_0}(\epsilon_n-\epsilon_n') f^{n-n_0-1})).
  \]
  Since \(f\) lies in the radical of \(R_I^\wedge\), the difference is nonzero in \(R_I^\wedge/\Frak{p}\).
  This proves the claim.
  Since \(\#\kappa(\Frak{p})\ge \#\{0,1\}^\nat=\Frak{c}\), we have \(\trdeg_\intg \kappa(\Frak{p})\ge \Frak{c}\), a contradiction.
\end{proof}

\begin{lem}\label{lem:open_closed_decomposition}
  Let \(D,E\in\CAlg(\pr)\), and let \(I\subset D\) be a smashing ideal.
  Assume that, for \(x,y\in E\), \(x\otimes y\cong 0\) implies \(x\cong 0\) or \(y\cong 0\).
  Then
  \[
  \Map_{\CAlg(\pr)}(D,E)\cong \Map_{\CAlg(\pr)}(D/I,E)\coprod \Map_{\CAlg(\pr)}(I,E).
  \]
\end{lem}
\begin{proof}
  Take \(F:D\to E\) in \(\CAlg(\pr)\).
  By the universal property of the quotient, it suffices to prove that \(F\) factors through exactly one of \(D/I\) and \(I\).
  Let \(e\to 1_D\) and \(1_D\to f\) be the left and right idempotents corresponding to \(I\).
  Since \(e\otimes f\cong0\), we have \(F(e)\cong 0\) or \(F(f)\cong 0\).
  If \(F(e)\cong 0\), then \(F\) factors through \(D/I\), and it does not factor through \(I\).
  The case \(F(f)\cong 0\) is similar.
\end{proof}

\begin{lem}\label{lem:smashing_descent_of_points}
  Let \(D,E\) be as in Lemma \ref{lem:open_closed_decomposition}.
  Let \(I_1,I_2\subset D\) be smashing ideals, let \(I_3=\oang{I_1,I_2}\), and let \(e_i\to 1_D\) be the corresponding left idempotents for \(i=1,2,3\).
  Assume that \(e_1\otimes e_2\cong 0\).
  Then the following diagram is coCartesian:
  \[
  \xymatrix{
    \Map_{\CAlg(\pr)}(D/I_3,E)\ar[r]\ar[d]&\Map_{\CAlg(\pr)}(D/I_2,E)\ar[d]\\
    \Map_{\CAlg(\pr)}(D/I_1,E)\ar[r]&\Map_{\CAlg(\pr)}(D,E).
  }
  \]
\end{lem}
\begin{proof}
  Let \(F:D\to E\).
  By the universal property of the quotient, it suffices to prove that \(F\) factors through \(D/I_1\) or \(D/I_2\).
  Let \(1_D\to f_2\) be the right idempotent corresponding to \(I_2\).
  The assumption \(e_1\otimes e_2\cong 0\) implies that \(e_1\otimes f_2\cong e_1\).
  Thus \(\iota_{I_1}^\mR=-\otimes e_1:D\to I_1\) factors through \(L_{I_2}=-\otimes f_2:D\to D/I_2\).
  The assertion follows from Lemma \ref{lem:open_closed_decomposition}.
\end{proof}

Let \(k\) be a ring.
Let \(\fld_k\) be the full subcategory of \(\CAlg_k(\pr)\) whose objects are \(D(K)\), where \(K\) is a field of transcendence degree \(<\Frak{c}\) over \(k\).
For a Noetherian formal scheme \(\Frak{X}\) over \(k\), let \((\fld_k)_{\eNuc(\Frak{X})/}\) be the comma category of \(\eNuc(\Frak{X})\) and \(\fld_k\) in \(\CAlg_k(\pr)\).

For an \(\infty\)-category \(C\), we denote by \(\pi_0(C)\) the class of connected components.

\begin{thm}\label{thm:points_reconstruction}
  Let \(k\) be \(\intg\) or a field, and let \(\Frak{X}\) be a Noetherian formal scheme over \(k\).
  Then the map of sets
  \[
  \phi:\{x\in \Frak{X}\mid \trdeg_k(\kappa(x))<\Frak{c}\}\to \pi_0((\fld_k)_{\eNuc(\Frak{X})/}): x\mapsto (\eNuc(\Frak{X})\to \eNuc(\kappa(x))\cong D(\kappa(x)))
  \]
  is bijective.
  If \(\Frak{X}\in\smFormalsch_k\), the same assertion holds for \(\csNuc(\Frak{X})\).
\end{thm}
\begin{proof}
Injectivity is clear.
We prove surjectivity.
By Lemma \ref{lem:smashing_descent_of_points}, we may assume that \(\Frak{X}=\Spf(R_I^\wedge)\) is affine.
Take \(F:\eNuc(\Frak{X})\to D(K)\) in \((\fld_k)_{\eNuc(\Frak{X})/}\).
It induces a homomorphism
\[
R\to R_I^\wedge=\End(1_{\eNuc(\Frak{X})})\to \End(1_{D(K)})=K.
\]
Let \(\Frak{p}\subset R_I^\wedge\) be the kernel of the induced map \(R_I^\wedge\to K\).
By Lemma \ref{lem:trancendental_degree_of_residue_field}, we have \(I\subset \Frak{p}\).
By Lemma \ref{lem:quotient_of_nuclear_modules}, we have \(\eNuc(\Frak{X})\otimes_{D(R_I^\wedge)}D(\kappa(\Frak{p}))\cong D(\kappa(\Frak{p}))\), and \(F\) factors through this category (in fact, it is an initial object of the connected component of \((\fld_k)_{\eNuc(\Frak{X})/}\)).
The same argument applies to \(\csNuc\).
\end{proof}
In particular, if \(\Frak{X}\) is nearly of \(<\Frak{c}\)-type, then the set of points of \(\Frak{X}\) is reconstructed.

\begin{cor}
  Let \(k\) be \(\intg\) or a field, and let \(\Frak{X}\) be a Noetherian formal scheme over \(k\).
  Then
  \[
  \Dtors(\Frak{X})^\perp\subset \bigcap_{F\in (\fld_k)_{\eNuc(\Frak{X})/}} \Ker(F).
  \]
  If \(\Frak{X}\) is nearly of \(<\Frak{c}\)-type over \(k\), then equality holds.
  If \(\Frak{X}\in \smFormalsch_k\), the same assertion holds for \(\csNuc(\Frak{X})\).
\end{cor}
\begin{proof}
If \(F,G\in (\fld_k)_{\eNuc(\Frak{X})/}\) lie in the same connected component, then \(\Ker(F)=\Ker(G)\).
Thus we may regard \(F\) as running over \(x^*:\eNuc(\Frak{X})\to D(\kappa(x))\) for points \(x\in\Frak{X}\) such that \(\trdeg_k\kappa(x)<\Frak{c}\).
Each \(x^*\) factors through \(\Dtors(\Frak{X})\), and \(\iota^\mR(\Dtors(\Frak{X})^\perp)\cong 0\), where \(\iota:\Dtors(\Frak{X})\to \eNuc(\Frak{X})\) denotes the inclusion.
This proves the inclusion.

Assume that \(\Frak{X}\) is nearly of \(<\Frak{c}\)-type.
For any nonzero object \(F\in \eNuc(\Frak{X})/\Dtors(\Frak{X})^\perp\cong \Dtors(\Frak{X})\), there exists \(x\in \Frak{X}\) such that \(x^*(F)\neq0\).
Hence equality holds.
The same argument applies to \(\csNuc\).
\end{proof}

\begin{thm}\label{thm:fully_faithfulness_strong_ver}
  Let \(k\) be \(\intg\) or a field, and let \(\Frak{X},\Frak{Y}\) be Noetherian formal schemes over \(k\).
  Assume that \(\Frak{Y}\) is nearly of \(<\Frak{c}\)-type over \(k\).
  Then
  \[
  \Map_{\CAlg_k(\dualpr)}(\eNuc(\Frak{X}),\eNuc(\Frak{Y}))\cong \Hom(\Frak{Y},\Frak{X}).
  \]
  Moreover, if \(\Frak{X}\in\smFormalsch_k\), then
  \[
  \Map_{\CAlg_k(\dualpr)}(\csNuc(\Frak{X}),\eNuc(\Frak{Y})) \cong \Hom(\Frak{Y},\Frak{X}).
  \]
  If \(\Frak{X},\Frak{Y}\in\smFormalsch_k\), then
  \[
  \Map_{\CAlg_k(\dualpr)}(\csNuc(\Frak{X}),\csNuc(\Frak{Y}))\cong\Hom(\Frak{Y},\Frak{X}).
  \]
\end{thm}
\begin{proof}
  We prove the first equivalence.
  Let \(F\in \Map_{\CAlg_k(\dualpr)}(\eNuc(\Frak{X}),\eNuc(\Frak{Y}))\).
  By Corollary \ref{cor:fully_faithful_weak_ver}, it suffices to prove that the composition \(\eNuc(\Frak{X})\os{F}{\to} \eNuc(\Frak{Y})\to \Dtors(\Frak{Y})\) factors through \(\Dtors(\Frak{X})\), or equivalently, that \(\Dtors(\Frak{X})^\perp\subset F^{-1}(\Dtors(\Frak{Y})^\perp)\).
  Precomposing with \(F\) defines a functor \((\fld_k)_{\eNuc(\Frak{Y})/}\to (\fld_k)_{\eNuc(\Frak{X})/}\).
  Thus,
  \[
  \Dtors(\Frak{X})^\perp\subset \bigcap_{G\in(\fld_k)_{\eNuc(\Frak{X})/}}\Ker(G)\subset F^{-1}\left(\bigcap_{H\in (\fld_k)_{\eNuc(\Frak{Y})/}}\Ker(H)\right)=F^{-1}\left(\Dtors(\Frak{Y})^\perp\right).
  \]
  The same argument proves the remaining assertions.
\end{proof}

\appendix

\section{Rigidification and smashing localization}

The purpose of this appendix is to prove Proposition \ref{prop:descent_for_eNuc}.
This appendix can be read independently of the main body of the paper.

By the universal property of rigidification, \(\eNuc(\Frak{X})\) (Definition \ref{defn:eNuc}) satisfies Zariski descent in the \(\infty\)-category of rigid tt-\(\infty\)-categories.
In fact, the restriction functors are smashing localizations, and the same descent statement holds in \(\pr\).
We now prove these statements.

\begin{lem}\label{lem:characterization_of_smashing_localizations}
  Let \(F:C\to D\) be a homomorphism of tt-\(\infty\)-categories.
  Assume that \(F\) is essentially surjective and that there exists a right idempotent \(\gamma:1_C\to f\) in \(C\) such that \(F^\mR F(x)\cong x\otimes f\) naturally in \(x\in C\), and such that, under this equivalence, the unit map \(x\to F^\mR F(x)\) is identified with \(x\otimes\gamma:x\to x\otimes f\).
  Then \(F\) is a quotient functor and its kernel is a smashing ideal.
\end{lem}
\begin{proof}
  Let \(I\) be the smashing ideal corresponding to \(\gamma\).
  Since \(F\) is essentially surjective, \(F^\mR(a)\cong 0\) implies \(a\cong 0\) for \(a\in D\).
  Hence \(F\) factors through the quotient \(C\to C/I\).
  After replacing \(C\) by \(C/I\), we may assume that \(\gamma=\id_{1_C}\).
  Then \(F^\mR F=\Id_C\), and hence \(F\) is fully faithful.
  Since \(F\) is essentially surjective by assumption, it is an equivalence.
\end{proof}

\begin{lem}\label{lem:smashing_ideal_preserves_RHom}
  Let \(C\) be a tt-\(\infty\)-category.
  Let \(I\) be a smashing ideal in \(C\).
  Then, for any \(x,y\in C\), there is a natural isomorphism \(\RHom_{I}(\iota_I^\mR(x),\iota_I^\mR(y))\cong \iota_I^\mR\RHom_{C}(x,y)\).
\end{lem}
\begin{proof}
  We denote the hom spectrum by \(\sHom\). 
  For any \(z\in I\), we have
  \begin{align*}
  \sHom(z,\iota_I^\mR\RHom_{C}(x,y))
  &\cong \sHom(\iota_I(z),\RHom_C(x,y))\\
  &\cong \sHom(\iota_I(z)\otimes x,y)\\
  &\cong \sHom(\iota_I(z\otimes \iota_I^\mR(x)),y)\\
  &\cong \sHom(z\otimes \iota_I^\mR(x),\iota_I^\mR(y))\\
  &\cong \sHom(z,\RHom_I(\iota_I^\mR(x),\iota_I^\mR(y))).
  \end{align*}
  The assertion follows by the Yoneda lemma.
\end{proof}

For an object \(x\) in a category \(C\), we denote the corresponding object in the opposite category \(C^\op\) by \(x^\op\).

\begin{prop}\label{prop:smashing_localization_of_nuclear_modules}
  Let \(C\) be a small symmetric monoidal stable \(\infty\)-category, let \(\gamma:1_C\to f\) be a right idempotent in \(C\), let \(I=\{x\in C\mid x\otimes f\cong 0\}\) be the corresponding thick ideal, let \(L:C\to C/I\) be the quotient functor, and let \(\Nuc(\Ind(L)):\Nuc(\Ind(C))\to \Nuc(\Ind(C/I))\) be the induced functor.
  Assume that, for every \(x\in C\), the object \(\RHom_{\Ind(C^{\op})}(x^\op,1)\) is nuclear in \(\Ind(C^{\op})\).
  Then \(\Nuc(\Ind(L))\) is a quotient functor and its kernel is a smashing ideal.
  The corresponding right idempotent is \(\nu_C(f)\), where \(\nu_C\) is the right adjoint of the inclusion \(\Nuc(\Ind(C))\to \Ind(C)\).
\end{prop}
\begin{proof}
  By \cite[Proposition 1.34]{Efimov2025LocalizingInvariantsInverseLimits}, \(\nu_C\) is symmetric monoidal and continuous.
  Let \(\iota_{C}:\Nuc(\Ind(C))\to \Ind(C)\) and \(\iota_{C/I}:\Nuc(\Ind(C/I))\to \Ind(C/I)\) be the inclusions, and let \(\nu_{C/I}\) be the right adjoint of \(\iota_{C/I}\).
  For any \(x\in \Nuc(\Ind(C))\) and \(a\in C\), we have \(\nu_C(\iota_C(x)\otimes a)\cong \nu_C\iota_C(x)\otimes \nu_C(a)\cong x\otimes \nu_C(a)\).
  For any \(x\in \Nuc(\Ind(C))\), we have
  \[
  \Nuc(\Ind(L))^\mR\Nuc(\Ind(L))(x)=\Nuc(\Ind(L))^\mR\nu_{C/I}\iota_{C/I}\Nuc(\Ind(L))(x)=\nu_C \Ind(L^\mR) \Ind(L)\iota_C(x)= x\otimes \nu_{C}(f).
  \]
  By Lemma \ref{lem:characterization_of_smashing_localizations}, it remains to prove that \(\Nuc(\Ind(L))\) is essentially surjective.
  There is a natural equivalence \(\Nuc(\Ind(C))\cong \Nuc(\Ind(C^{\op}))\), and the composition
  \[
  C\subset \Ind(C)\os{\nu_C}{\to} \Nuc(\Ind(C))\cong \Nuc(\Ind(C^{\op}))\subset \Ind(C^{\op})
  \]
  is given by \(x\mapsto \RHom_{\Ind(C^{\op})}(x^\op,1)\) (\cite[Proposition 1.34]{Efimov2025LocalizingInvariantsInverseLimits}).
  The analogous statement holds for \(C/I\).
  The inclusion \(\Ind((L^\mR)^\op):\Ind((C/I)^{\op})\to\Ind(C^{\op})\) is identified with the inclusion of a smashing ideal, and the corresponding left idempotent is \(\gamma^\op:f^\op\to 1\).
  The functor \(\Ind(L^\op)\) is its right adjoint.
  By Lemma \ref{lem:smashing_ideal_preserves_RHom}, for any \(x\in C\), there is an isomorphism
  \[
  \RHom_{\Ind((C/I)^{\op})}(L(x)^\op,1_{\Ind((C/I)^{\op})})\cong \Ind(L^\op)(\RHom_{\Ind(C^{\op})}(x^\op,1_{\Ind(C^{\op})})).
  \]
  This isomorphism implies that \(\nu_{C/I}\Ind(L):\Ind(C)\to \Ind(C/I)\to \Nuc(\Ind(C/I))\) factors through \(\Nuc(\Ind(C))\).
  Hence \(\Nuc(\Ind(L))\) is essentially surjective.
\end{proof}

Let \(E\) be a locally rigid tt-\(\infty\)-category whose unit object \(1_E\) is \(\omega_1\)-compact.
By \cite[Theorem 4.2]{Efimov2025LocalizingInvariantsInverseLimits}, the composition
\[
  E^\rig\subset \Ind(E^{\omega_1})\xrightarrow{\colim} E
\]
admits a symmetric monoidal right adjoint \(\Lambda_E\).

\begin{cor}\label{cor:rigidification_preserves_smashing_localizations}
  Let \(E\) be a locally rigid symmetric monoidal stable \(\infty\)-category with \(\omega_1\)-compact unit object \(1_E\), and let \(I\subset E\) be a smashing ideal with corresponding right idempotent \(\gamma: 1_E\to f\).
  Assume that \(f\) is \(\omega_1\)-compact.
  Then \(E^\rig\to (E/I)^\rig\) is a quotient functor and its kernel is a smashing ideal.
  The corresponding right idempotent is \(\Lambda_E(\gamma)\).
\end{cor}

\begin{prop}\label{prop:descent_for_rigidification}
  Let \(E\) be a locally rigid symmetric monoidal stable \(\infty\)-category with \(\omega_1\)-compact unit object \(1_E\), and let \(I_1,I_2\subset E\) be smashing ideals with corresponding right idempotents \(\gamma_1: 1_E\to f_1\) and \(\gamma_2: 1_E\to f_2\).
  Assume that \(f_1\) and \(f_2\) are \(\omega_1\)-compact and that \(I_1\cap I_2=0\).
  Let \(I_3=\oang{I_1,I_2}\).
  Then the following diagram is Cartesian in \(\pr\):
  \[
  \xymatrix{
  E^\rig\ar[d]\ar[r]&(E/I_2)^\rig\ar[d]\\
(E/I_1)^\rig\ar[r]&(E/I_3)^\rig.
  }
  \]
\end{prop}
\begin{proof}
  Let \(e_i\to 1_E\) and \(1_E\to f_i\) be the left and right idempotents corresponding to \(I_i\), and let \(L_i\colon E^\rig\to (E/I_i)^\rig\) be the induced functor, for \(i=1,2,3\).
  By Corollary \ref{cor:rigidification_preserves_smashing_localizations}, these functors are quotient functors and their kernels are smashing ideals.
  The corresponding left and right idempotents are \(\Lambda_E(e_i)\) and \(\Lambda_E(f_i)\), respectively.
  Since \(\Lambda_E\) is symmetric monoidal, we have \(\Lambda_E(e_1)\otimes \Lambda_E(e_2)\cong 0\) and \(\Lambda_E(f_1)\otimes \Lambda_E(f_2)\cong \Lambda_E(f_3)\).
  For each \(i\), the localization \(L_i^\mR L_i\) is given by tensoring with the right idempotent \(\Lambda_E(f_i)\).
  Therefore \(L_1^\mR L_1\circ L_2^\mR L_2\cong L_2^\mR L_2\circ L_1^\mR L_1\cong L_3^\mR L_3\).
  Moreover, if \(L_1^\mR L_1(x)\cong 0\) and \(L_2^\mR L_2(x)\cong 0\), then \(x\cong \Lambda_E(e_1)\otimes x\) and \(x\cong \Lambda_E(e_2)\otimes x\).
  Hence \(x\cong x\otimes \Lambda_E(e_1)\otimes \Lambda_E(e_2)\cong 0\).
  Thus the hypotheses of \cite[Proposition 10.5]{Scholze2026LecturesCondensedMathematics} are satisfied, and the assertion follows.
\end{proof}

\section*{Acknowledgments}
I am deeply grateful to Seidai Yasuda for his invaluable support during the preparation of this paper.

This work is supported by JSPS KAKENHI Grant Number JP26KJ0476.

\printbibliography
\end{document}